\documentclass[11pt, reqno]{amsart}
\usepackage{latexsym, amsmath, amssymb, longtable, booktabs,amscd, color}

\usepackage[english]{babel}
\usepackage[utf8]{inputenc}
\usepackage{mathtools}
\usepackage{ stmaryrd }

\usepackage[all]{xy}

\numberwithin{equation}{section}
\newcommand\N{\mathbb{N}}

\newcommand\GL{{\mathrm{GL}}}

\newcommand{\Q}{\mathbb{Q}}
\newcommand{\Z}{\mathbb{Z}}

\newcommand{\F}{\mathbb{F}}

\newcommand\Gal{{\mathrm {Gal}}}

\newcommand\Sym{{\mathrm {Sym}}}

\newcommand\iso{{\> \simeq \>}}
 
\newcommand\linee{{
        {\begin{tiny}
        \begin{xymatrix}{
         \bullet  \ar@{-}[d] \\ 
         \bullet  \ar@{-}[d] \\
         \bullet} 
       \end{xymatrix}
       \end{tiny}} }}  
\newcommand\diamondd{{
        \begin{tiny}
        \begin{xymatrix}{
         & \bullet  \ar@{-}[dl] \ar@{-}[dr] &  \\ 
         \bullet  \ar@{-}[dr] &   &   \bullet  \ar@{-}[dl] \\
         & \bullet & } 
       \end{xymatrix}
       \end{tiny} }}

\newtheorem{thm}{Theorem}[section]

\newtheorem{prop}[thm]{Proposition}
\newtheorem{lemma}[thm]{Lemma}
\theoremstyle{definition}

\theoremstyle{remark}
\newtheorem{remark}[thm]{Remark}

\theoremstyle{definition}

\theoremstyle{remark}

\theoremstyle{remark}

\title[]{Reduction of certain crystalline representations and local constancy in the weight space}

\author{Shalini Bhattacharya}
\address{Max Planck Institute for Mathematics, Bonn}
\email{shalini@mpim-bonn.mpg.de, shaliniwork16@gmail.com}

\begin{document}

\maketitle

\section{Introduction}
Let $p$ be an odd prime number. Let $E$ be a finite extension of $\Q_p$ and let $v:\bar\Q_p^*\rightarrow\Q$ 
be the normalized valuation
so that $v(p)=1$. Let $ m_E$ be the maximal ideal in the ring of integers $\mathcal{O}_E$ of $E$.
 For any integer $k\geq 2$ and any $a_p\in m_E$, let $D_{k,a_p} = Ee_1 \oplus Ee_2$ be the filtered 
 $\varphi$-module, where the Frobenius operator $\varphi$ is 
 given by the matrix $\left(\begin{smallmatrix} 0 & -1\\ p^{k-1} & a_p\end{smallmatrix}\right)$
 with respect to the basis $\langle e_1, e_2\rangle$,
 and the filtration  is given by
 $$\mathrm{Fil}^i(D_{k,a_p})=\begin{cases}
                                                     Ee_1\oplus Ee_2, &\text{if } i\leq 0\\
                                                     Ee_1, &\text{if } 1\leq i\leq k-1\\
                                                     0, &\text{if } k\leq i.
                                                     \end{cases}$$
  Now let $V=V_{k,a_p}$ be the unique two-dimensional
irreducible crystalline representation of $G_{Q_p}:=\mathrm{Gal}(\bar\Q_p|\Q_p)$ such that $\mathrm{D_{cris}}(V^*) = D_{k,a_p}$, where $V^*$ denotes the dual representation of $V$. The existence of such representations follows from the theory of Colmez and Fontaine \cite{Colmez-Fontaine}. We recall that $V_{k,a_p}$ has  Hodge-Tate weights $(0,k-1)$ and slope $v(a_p)>0$.

 The semisimplification  of the mod $p$ reduction $\bar V_{k,a_p}$ with respect to a $G_{\Q_p}$-stable integral lattice in $V_{k,a_p}$ 
is independent of the choice of the lattice. 
 Despite the variety of non-isomorphic irreducible two-dimensional crystalline representations $V$ in characteristic 0 which are indexed by the tuples $(k,a_p)$ up to twists,
 one has very limited choice for their semisimplified reductions $\bar V$ at the mod $p$ level. 
 The behaviour of the mod $p$ reductions of these objects $V_{k,a_p}$ has been studied by several mathematicians.
 The explicit shape of $\bar V_{k,a_p}$ has been computed only for small weights $k \leq 2p+1$ \cite{Edixhoven92, [Br03]},  small slopes $v(a_p)<2$
 \cite{[BG09], Buzzard-Gee13, [GG15], Bhattacharya-Ghate, Bhattacharya-Ghate-Rozensztajn16}, 
 or when the slope is very large compared to the weight $k$ \cite{Berger-Li-Zhu04}. For small values of $k$ and $p$, now one can also compute these reductions using the algorithm given in \cite{[Roz17]}. 
 
 In this article we attempt to study how the reduction behaves with varying weight $k$, where $a_p\in m_E$ is kept constant.
 Let us begin by recalling the following result about the local constancy of the map 
 $k\mapsto \bar V_{k,a_p}$, for any fixed non-zero $a_p\in m_{\bar\Z_p}$.
 
 \begin{thm}[Thm B,\cite{Berger12}]\label{Berger} Let $\alpha(r) := \underset{n\geq 1}\sum \lfloor r/p^{n-1}(p-1)\rfloor$, for any $r\in\Z$.\\
 If $a_p\neq 0$ and $k > 3v(a_p)+\alpha(k-1)+1$, then there exists $m=m(k,a_p)$ such that $\bar V_{k',a_p} \cong\bar V_{k,a_p}$,
  if $k'\geq k$ and $k'-k\in p^{m-1}(p-1)\Z$.
 \end{thm}
 In the context of the theorem above, one may ask the following questions:
 \begin{enumerate}
 \item Can one improve the lower bound $3v(a_p)+\alpha(k-1)+1$ on $k$?
 \item What are the possible values of the constant $m(k,a_p)$? \\
 For  fixed $a_p$,  is it possible to choose an $m(k,a_p)$ that works for all $k$?
 This phenomenon, when occurs, can be referred to as "uniform local constancy" of the reduction.
 \end{enumerate}
 One may expect uniform local constancy to hold generically, as it is true for small slopes,
 where the reductions have been explicitly computed. 
 Let us recall the cases where we know explicit (smallest) values of $m(k,a_p)$: 
 \begin{itemize}
 \item $v(a_p)\in(0,1)$: 
 We have $m(k,a_p)=1$, cf. \cite{[BG09]},\\
 unless $k\equiv 3\mod (p-1)$ and $v(a_p)=1/2$. 
 For $v(a_p)=1/2$ and $k\equiv 3\mod (p-1)$,  the behaviour of the reduction is complicated, and
 it is clear from the main theorem of \cite{Buzzard-Gee13} that $m(k,a_p)$ depends on $k$ in a more serious way.
 \item $v(a_p)=1$:
 We have $m(k,a_p)=\begin{cases}3, &\text{if } k\equiv 3\mod (p-1)\\ 
 2, &\text{if } k\not\equiv 3, 4\mod (p-1).\end{cases}$\\
 For  $k\equiv 4\mod (p-1)$, the reductions are more complex \cite{Bhattacharya-Ghate-Rozensztajn16}.
 \item $v(a_p)\in (1,2)$:
 We have $m(k,a_p)=\begin{cases}3, &\text{if } k\equiv 3\mod (p-1)\\ 2, &\text{if } k\not\equiv 3\mod (p-1),\end{cases}$\\
 unless $v(a_p)=3/2$ and $k\equiv 5\mod (p-1)$.
 For the remaining exceptional case, i.e., when $v(a_p)=3/2$ and $k\equiv 5\mod (p-1)$,  we refer to the ongoing work \cite{[GR18]}.
 \end{itemize}
 We notice that the value of $m(k,a_p)$ does not depend on $k$ in most cases, 
 but it does increase with the slope $v(a_p)$ in general.
 
 In this article 
 we compute $m(k,a_p)$ for some small values of $k$. We also improve the lower bound on $k$ a bit.
 However we could not avoid  a lower bound that is linear in $v(a_p)$, as in Theorem \ref{Berger} by Berger.
 More precisely, we prove that
 \begin{thm}\label{cons}
 If $2v(a_p)+2 < k \leq p+1$, then $m(k,a_p)\leq 2v(a_p)+1$. 
 
 Under the extra assumption 
 $\frac{a_p}{p^{(k-2)/2}}\not\equiv\pm 1\mod \wp$, 
 the same is true for odd weights $k=2v(a_p)+2\leq p+1$.
 \end{thm}
\begin{remark}
(a) Note that the hypothesis of Theorem \ref{cons} forces that $a_p\neq 0$. 
In fact, it follows from Prop. 4.1.4 in \cite{Berger-Li-Zhu04} that there is no local constancy  
with respect to weight at $a_p=0$.

(b) For $k\leq p-1$ (or $k=p, p+1$), the constant $m(k,a_p)$ has been proved to exist only under the condition 
$k>3v(a_p)+1$ (or $k>3v(a_p)+2$). 
Direct computation gives us a slightly better lower bound $2v(a_p)+2$ on $k$. 
 However it is not clear if this is a strict bound or why such a bound should be necessary at all. 
\end{remark}
Let $G_{\Q_{p^2}}$ denote the subgroup $\mathrm{Gal}(\bar\Q_p|\Q_{p^2})$ of index $2$ in $G_{\Q_p}$, where $\Q_{p^2}$ is the unique quadratic unramified extension of $\Q_p$. Let $\omega:G_p\rightarrow \bar\F_p^* $ and $\omega_2:G_{\Q_{p^2}}\rightarrow \bar\F_p^*$ denote the fundamental characters of level one and two respectively. 

Theorem \ref{cons} is an easy corollary of the following, which is the main result of this article.
 \begin{thm}
 Let $k'\equiv k \mod (p-1)$, for some $2v(a_p)+2 < k \leq p+1$. 
 If $t=v(k'-k)\geq 2v(a_p)$, then $\bar V_{k',a_p}$ is 
  irreducible of the form $\mathrm{ind}_{G_{\Q_{p^2}}}^{G_{\Q_p}}(\omega_2^{k-1})$.
 \end{thm}
 This shows that if $k'$ is close enough to $k$ in the weight space with an explicit upper bound on their distance that is linear in
 $v(a_p)$, then $\bar V_{k',a_p}$ is isomorphic to $\bar V_{k,a_p}$.
 Based on the known results for slopes $<2$, one hopes that this
 upper  bound for 
 $m(k,a_p)$ should work for almost all $k$ and not just for $k\leq p+1$.
 However, here our computations are limited to the weights $k'$ close to the small weights $k\leq p+1$, as stated above.
 
 The proof (of Thm. \ref{thmending}) uses the compatibility of $p$-adic and mod $p$ Local Langlands correspondences, following the method of \cite{[Br03], [BG09]}.
 We generalise some of the techniques introduced in \cite{Bhattacharya-Ghate} and \cite{Bhattacharya-Ghate-Rozensztajn16}.
 More details about the proof are given in the next section.
 
 \section{Basics}\label{basics}

In this section we quickly recall some notations and then explain the basic principle of our proof.
\subsection{The Hecke operator }
  \label{Hecke}

Let $G = \mathrm{GL}_2(\Q_p)$, $K = \mathrm{GL}_2(\Z_p)$ be the standard
maximal compact subgroup of $G$ and $Z \cong \Q_p^\times$ be the center of  the group $G$.
Let us begin by recalling the Hecke operator $T$ which acts $G$-linearly on the compact induction 
$\mathrm{ind}_{KZ}^G V$ for certain representations $V$ of $KZ$.

Let $R$ be a $\Z_p$-algebra and let $V = \Sym^r R^2\otimes D^s$ be the
usual symmetric power representation of $KZ$ twisted by a power of the
determinant character $D$, modeled on homogeneous polynomials of degree
$r$ in the variables $X$, $Y$ over $R$. For $g \in G$, $v \in V$, let
$[g,v] \in \mathrm{ind}_{KZ}^{G} V$ be the function with support in the coset
${KZ}g^{-1}$ given by 
  $$g' \mapsto
     \begin{cases}
         g'g \cdot v,  \ & \text{ if } g' \in {KZ}g^{-1} \\
         0,                  & \text{ otherwise.}
      \end{cases}$$ 
Any element of $\mathrm{ind}_{KZ}^G V$ is a $V$-valued function on $G$ that is compactly supported mod $KZ$ and thus is a finite linear combination of
functions of the form $[g,v]$, for $g\in G$ and $v\in V$.  The Hecke
operator  $T$ is defined by its action on these elementary functions via the formula

\begin{equation}\label{T} T([g,v(X,Y)])=\underset{\lambda\in\F_p}\sum\left[g\left(\begin{smallmatrix} p & [\lambda]\\
                                                 0 & 1
                                                \end{smallmatrix}\right),\:v\left(X, -[\lambda]X+pY\right)\right]+\left[g\left(\begin{smallmatrix} 1 & 0\\
                                                                                                                                                   0 & p
                                                \end{smallmatrix}\right),\:v(pX,Y)\right],\end{equation}
 where $[\lambda]$ denotes the Teichm\"uller representative of $\lambda\in\F_p$.

\subsection{The Local Langlands Correspondences}
\label{subsectionLLC}

Let $\Gamma $ denote the finite group $\GL_2(\F_p)$ which naturally acts on the two-dimensional vector space over $\bar\F_p$. For any $r\geq 0$, we have the symmetric power representations 
$$V_r:=\Sym^r\bar\F_p^2\in\mathrm{Rep}_{\,\bar\F_p}(\Gamma)$$ of dimension $r+1$. For $0 \leq r \leq p-1$, $\lambda \in \bar{\F}_p$ and $\eta : \Q_p^\times
\rightarrow \bar\F_p^\times$ a smooth character, we know that
\begin{eqnarray*}
  \pi(r, \lambda, \eta) & := & \frac{\mathrm{ind}_{KZ}^{G} V_r}{T-\lambda} \otimes (\eta\circ \mathrm{det})
\end{eqnarray*}
are smooth admissible representations of ${G}$, also irreducible in most cases. Recall that here $p\in KZ$ acts on $V_r:=\Sym^r\bar\F_p^2$ trivially and 
the rest of $KZ$ acts by the inflation of $K=\GL_2(\Z_p)\twoheadrightarrow \Gamma$.  
These objects $\pi(r,\lambda,\eta)$ together capture all possible irreducible representations of $G$ in characteristic $p$, as proved in \cite{Barthel-Livne94, BL, Breuil1}.

With this notation, Breuil's semisimple mod $p$ Local Langlands Correspondence 
\cite[Def. 1.1]{[Br03]} 
is given by:
\begin{itemize} 
  \item  $\lambda = 0$: \quad\quad\qquad
             $\mathrm{ind}_{G_{\Q_{p^2}}}^{G_{\Q_p}}(\omega_2^{r+1}) \otimes \eta  \:\:\overset{LL}\longmapsto\:\: \pi(r,0,\eta)$,
  \item $\lambda \neq 0$: \quad
             $\left( \mu_\lambda \omega^{r+1}  \oplus \mu_{\lambda^{-1}} \right) \otimes \eta 
                   \:\:\overset{LL}\longmapsto\:\:  \pi(r, \lambda, \eta)^{ss} \oplus  \pi([p-3-r], \lambda^{-1}, \eta \omega^{r+1})^{ss}$, 
\end{itemize}
where $\{0,1, \ldots, p-2 \} \ni [p-3-r] \equiv p-3-r \mod (p-1)$.

On the other hand, by the $p$-adic Local Langlands correspondence we have the association 
$V_{k,a_p}\rightsquigarrow \Pi_{k,a_p}$, where $\Pi_{k,a_p}$ is the locally algebraic representation of $G$ given by 
\begin{eqnarray*} 
\Pi_{k, a_p} = \frac{ \mathrm{ind}_{{KZ}}^{G}
\Sym^r \bar\Q_p^2 }{(T-a_p)}, 
\end{eqnarray*} 
where $r=k-2 \geq 0$ and $T$ is the
Hecke operator as usual. Consider the standard lattice in $\Pi_{k,a_p}$ given by
\begin{equation} 
\label{definetheta}  
\Theta_{k, a_p} :=
\mathrm{image} \left( \mathrm{ind}_{KZ}^{G} \Sym^r \bar\Z_p^2 \rightarrow
\Pi_{k, a_p} \right) \iso \frac{ \mathrm{ind}_{KZ}^{G} \Sym^r \bar\Z_p^2
}{(T-a_p)(\mathrm{ind}_{KZ}^{G} \Sym^r \bar\Q_p^2) \cap
\mathrm{ind}_{KZ}^{G} \Sym^r \bar\Z_p^2 }.  
\end{equation} 
By the commutativity of the $p$-adic and mod $p$ Local Langlands Correspondence, conjectured in \cite{[Br03]} and
 proved in \cite{Berger10}, we know that
$$\bar\Theta_{k,a_p}^\mathrm{ss} :=\Theta_{k,a_p}\otimes \bar\F_p\iso LL(\bar{V}_{k,a_p}^{ss}).$$ 
  Since the correspondence ${LL}$ at the mod $p$ level is injective, computing
   ${LL}(\bar V_{k,a_p}^{ss})$ is enough to determine 
$\bar{V}_{k,a_p}^{ss}$. Therefore, we are going to study $\bar\Theta_{k,a_p}^{ss}$
as an object in $\mathrm{Rep}_{\:\bar\F_p} (G)$. 
The superscript `$ss$' will often be omitted, as we are always concerned about the semisimplified reduction.
\section{Computations}
\subsection{Some results in characteristic $p$} Here we prove some general lemmas in characteristic $p$ that will be useful in computing the reduction $\bar\Theta_{k,a_p}$.

By the definition of $\bar\Theta_{k,a_p}$, we have a natural surjection
$$P:\mathrm{ind}_{KZ}^GV_r\twoheadrightarrow \bar\Theta_{k,a_p},$$ for $r=k-2$.

Consider the special polynomial 
$$\theta(X,Y):=X^pY-Y^pX=-X\cdot\underset{\lambda\in\F_p}\prod (Y-\lambda X)\in \Sym^{p+1}\bar\F_p^2=V_{p+1},$$
on which $\Gamma:=\GL_2(\F_p)$ acts by the determinant character.
Define for each $m\in \N$, 
$$V_r^{(m)}:=\{f\in V_r: \theta^m \text{ divides } f \text{ in } \bar\F_p[X,Y]\},$$ so that 
$V_r\supseteq V_r^{(1)}\supseteq V_r^{(2)}\supseteq\cdots$ is a chain of $\Gamma$-stable submodules of length 
$\lfloor \frac{r}{p+1}\rfloor +1$. Moreover, we know that $V_r^{(m)}\cong V_{r-m(p+1)}\otimes D^m$, 
where $D$ denotes the determinant character.
\begin{lemma}\label{divconds}
Let $F(X,Y)=\underset{i=0}{\overset{r}\sum} a_i X^{r-i}Y^i \in V_r $ be a polynomial such that 
$$\bar\F_p\ni a_i\neq 0 \implies i\equiv a\mod (p-1),$$
for some fixed congruence class $a\mod (p-1)$. 
Then for any $m\geq 0$, we have $F(X,Y)\in V_r^{(m)}$ if and only if the following conditions are satisfied:
\begin{itemize}
\item $a_i\neq 0\implies m\leq i\leq r-m$,
\item $\underset{i}\sum (j)! {i\choose j}a_i = 0\in \bar\F_p$, for $0\leq j\leq m-1$.
\end{itemize}
\end{lemma}
\begin{proof}
We consider $f(z)=\underset{i=0}{\overset{r}\sum} a_i z^i\in \bar\F_p[z]$, so that $F(X,Y)=X^r\cdot f(\frac{Y}{X})$. Note
\begin{eqnarray*}
\theta^m\mid F(X,Y) &\iff & F(X,Y)= (-X)^m\underset{\lambda\in \F_p}\prod (Y-\lambda X)^m F_1(X,Y), \quad F_1\in V_{r-(p+1)m},\\
                                     &\iff & X^m\mid F(X,Y) \text{ and }  
                                     f(Y/X)= \underset{\lambda\in \F_p}\prod (Y/X-\lambda )^m F_1(1,Y/X),\\
                                     &\iff& X^m\mid F(X,Y) \text{ and }  
                                     f(z)= \underset{\lambda\in \F_p}\prod (z-\lambda )^m f_1(z),\\
                                     &\iff& X^m, Y^m \mid F(X,Y) \text{ and }  
                                      (z-\lambda )^m \mid f(z), \quad \forall \lambda\in \F_p^*.
\end{eqnarray*}
The conditions $X^m, Y^m \mid F(X,Y) $ are equivalent to $a_i\neq 0\implies m\leq i\leq r-m$, and 
 $(z-\lambda)^m$ divides $f(z)$ if and only if $f(\lambda)=f'(\lambda)=\cdots =f^{(m-1)}(\lambda)=0\in\bar\F_p$.
 Looking at the coefficients of $f(z)$, for $\lambda\in \F_p^*$, we have 
 $$f^{(j)}(\lambda)=\underset{i} \sum a_i \cdot i(i-1)\cdots (i-j+1)\lambda^{i-j}=\lambda^{a-j}\cdot\underset{i}\sum a_i {i\choose j} j!,$$
using the hypothesis on the coefficients of $F(X,Y)$. This completes our proof.
 
 Note that as we are in the situation  $j< m\leq i\leq r-m$ here, the binomial coefficients ${i\choose j}$ above are all a priori non-zero,
 though some of them might vanish mod $p$. 
\end{proof}
For all integers $m\geq 0$ let us define the polynomials $F_m$ in $V_r$ as
\begin{equation}\label{F_m}F_m(X, Y):= X^mY^{r-m}-X^{r-b+m}Y^{b-m},\end{equation}
where $r\equiv b\mod (p-1)$, so that Lemma \ref{divconds} can be applied on
 $F_m$.
With this notation, we prove the following key lemma:

\begin{lemma}\label{polynomial}
Let 
 $t=v(r-b)\geq 1$ and let $m\geq 1$. 

(a) For $b\geq 2m$, the polynomial
$F_m(X, Y)$ is divisible by $\theta^m$ but not by $\theta^{m+1}$.

(b) For $b>2m$, the image of $F_m$ generates the subquotient 
$\frac{V_r^{(m)}}{V_r^{(m+1)}}$ over $G$. 
\end{lemma}
\begin{proof}
(a) Any polynomial divisible by $\theta^{m+1}$ is a multiple of $X^{m+1}$, so  $\theta^{m+1}\nmid F_m$.

To show $\theta^m\mid F_m$, by Lemma \ref{divconds} we need to show both $m, b-m\geq m$, and further for all $0\leq j\leq m-1$, 
$$j!\left({r-m\choose j}-{b-m\choose j}\right)=0\mod p, $$
	which is ensured by the fact $t=v(r-b)\geq 1$.
Note that the last condition is satisfied for $j=m$ as well.\\

(b) 
We first claim that the polynomial
$$H_m(X,Y):=F_m(X,Y)-(-1)^m\theta(X,Y)^m(Y^{r-m(p+1)}-Y^{b-2m}X^{r-b-pm+m})\in V_r$$ 
lies in the submodule $V_r^{(m+1)}$. 

Assuming the claim, it is enough to show the image of  $\theta^m(Y^{r-m(p+1)}-Y^{b-2m}X^{r-b-pm+m})$ 
generates $\dfrac{V_r^{(m)}}{V_r^{(m+1)}}\cong D^m\otimes \dfrac{V_{r-m(p+1)}}{ V_{r-m(p+1)}^{(1)}}
\cong D^m\otimes \dfrac{V_{b-2m+p-1}}{ V_{b-2m+p-1}^{(1)}}$. From Lemma 5.3, \cite{[Br03]}, we obtain the short exact sequence
$$0\rightarrow D^m\otimes V_{b-2m}\rightarrow D^m\otimes \dfrac{V_{b-2m+p-1}}{ V_{b-2m+p-1}^{(1)}}
\rightarrow D^{b-m}\otimes V_{p-1-(b-2m)}\rightarrow 0,$$ 
which does not split as $0<b-2m\leq b-2<p-1$ (Prop. 2.1, \cite{Bhattacharya-Ghate}). Hence it is enough to show that the image of
$\theta^m( Y^{r-m(p+1)}-Y^{b-2m}X^{r-b-pm+m})$ maps to a non-zero element in the quotient above. 
We check that in fact its image 
$$Y^{b-2m+p-1}-Y^{b-2m}X^{p-1}\in D^m\otimes \dfrac{V_{b-2m+p-1}}{V_{b-2m+p-1}^{(1)}}$$ maps to 
$-X^{p-1-(b-2m)}\in D^{b-m}\otimes V_{p-1-(b-2m)}$.

\underline{Proof of claim}: The lowest degree of $X$ in  $H_m(X,Y)$ is $\geq m+p-1\geq m+1$, and the lowest degree of $Y$ in $H_m(X,Y)$ is $\geq b-m\geq m+1$, as $b>2m$ by hypothesis. 
Following the proof of Lemma \ref{divconds}, we consider
\begin{eqnarray*}
h_m(z):=H_m(1,z)&=&z^{r-m}-z^{b-m}-(-1)^m(z-z^p)^m(z^{r-m(p+1)}-z^{b-2m})\\
        &=& z^{r-m}-z^{b-m}-(z^{p-1}-1)^m(z^{r-mp}-z^{b-m}).
\end{eqnarray*}
We already know $X^{m+1}, Y^{m+1}$ divide $H_m(X,Y)$, hence
$$\theta^{m+1}\mid H_m(X,Y)
\iff (z-\lambda)^{m+1}\mid h_m(z), \quad \forall \lambda\in \F_p^*.$$
Equivalently, we need $\frac{d^ih_m}{dz^i}(\lambda)=0$  for all $0\leq i\leq m$, and all $\lambda\in \F_p^*$.
For the first part $F_m(1,z)=z^{r-m}-z^{b-m}$ of $h_m(z)$, this vanishing of derivatives is already proved in part (a) above. For the other part 
$-(z^{p-1}-1)^m(z^{r-mp}-z^{b-m})$ of $h_m(z)$, the derivatives up to order $m$ vanish since $1-\lambda^{p-1}=0=\lambda^{r-mp}-\lambda^{b-m}$,
for all $\lambda\in \F_p^*$.
\end{proof}

Now we recall a very useful fact from Remark 4.4, \cite{[BG09]}, that if
$v(a_p)<m$ and $r= k-2 \geq m(p+1)$, then $\bar\Theta_{k,a_p}$ is a quotient of
$\mathrm{ind}_{KZ}^G (V_r/V_r^{(m)})$. 
We fix an $a_p$ with positive  valuation, and let $n\in\N$ be the smallest such that $v(a_p)<n+1$, so we have  
\begin{equation}\label{surj}P:\mathrm{ind}_{KZ}^G (V_r/V_r^{(n+1)})\twoheadrightarrow \bar\Theta_{k,a_p}.\end{equation}
We consider the chain of submodules of length $n+1$
$$0\subseteq \frac{V_r^{(n)}}{V_r^{(n+1)}}\subseteq \frac{V_r^{(n-1)}}{V_r^{(n+1)}}\subseteq\cdots\subseteq \frac{V_r}{V_r^{(n+1)}}, $$ inducing 
$$0\subseteq M_n=\mathrm{ind}_{KZ}^G\left(\frac{V_r^{(n)}}{V_r^{(n+1)}}\right)\subseteq M_{n-1}=\mathrm{ind}_{KZ}^G\left(\frac{V_r^{(n-1)}}{V_r^{(n+1)}}\right)\subseteq\cdots\subseteq M_0=\mathrm{ind}_{KZ}^G \left(\frac{V_r}{V_r^{(n+1)}}\right), $$
with  respective images
\begin{equation}\label{chain}P(M_n)\subseteq P(M_{n-1})\subseteq\cdots \subseteq P(M_0):=\bar\Theta_{k,a_p}
\end{equation}
inside $\bar\Theta_{k,a_p}$. 
We have this chain of submodules inside $\bar\Theta_{k,a_p}$, and we will try to compute it piece by piece.
For example, we would like to check if some of the quotient factors in the chain above are in fact zero in $\bar\Theta_{k,a_p}$.

\subsection{Computations in characterstic 0:}
We extend the formula for the Hecke operator $T$ when acting on $\mathrm{ind}_{KZ}^{G}\Sym^r\bar\Q_p^2$ in particular, 
 to see how $T$ acts on its explicit elements viewed as $\Sym^r\bar\Q_p^2$-valued functions on the 
Bruhat-Tits tree for $\GL_2$.
 
For $m = 0$, set $I_0 = \{0\}$, and 
for $m >0$, let
$I_m = \{ [\lambda_0] + [\lambda_1] p + \cdots + [\lambda_{m-1}]p^{m-1}  \> : \>  \lambda_i \in \F_p \} 
              \subset \Z_p$,
where the square brackets denote Teichm\"uller representatives. For $m \geq 1$, 
there is a truncation map
$[\quad]_{m-1}: I_{m} \rightarrow I_{m-1}$ given by taking the first $m-1$ terms in the $p$-adic expansion above;
for $m = 1$, $[\quad]_{m-1}$ is the $0$-map.
Let $\alpha =  \left( \begin{smallmatrix} 1 & 0 \\ 0  & p \end{smallmatrix} \right)$. 
For $m \geq 0$ and $\lambda \in I_m$, let
\begin{eqnarray*}
  g^0_{m, \lambda} =  \left( \begin{smallmatrix} p^m & \lambda \\ 0 & 1 \end{smallmatrix} \right) & \quad \text{and} \quad 
  g^1_{m, \lambda} = \left( \begin{smallmatrix} 1 & 0  \\ p \lambda  & p^{m+1} \end{smallmatrix} \right),
\end{eqnarray*}
noting that $g^0_{0,0}=\mathrm{Id}$ is the identity matrix and $g_{0,0}^1=\alpha$ in $G$. 
We have 
    $$G  = \coprod_{\substack{m\geq 0,\,\lambda \in I_m,\\ i\in\{0,1\}}} {KZ} (g^i_{m, \lambda})^{-1}.$$
Thus a general element in $\mathrm{ind}_{KZ}^G V$ is a finite sum of functions of the form $[g,v]$, 
with $g=g_{m,\lambda}^0$ or $\,g_{m,\lambda}^1$, for some $\lambda\in I_m$ and $v\in V$.
For a $\Z_p$-algebra $R$, let $v = \sum_{i=0}^r c_i X^{r-i} Y^i \in V = \mathrm{Sym}^r R^2\otimes D^s$.
Expanding the formula \eqref{T} for the Hecke operator $T$ one may write
$T  = T^+ + T^-$, with 
\begin{eqnarray}
 \label{T^+}  T^+([g^0_{n,\mu},v]) & = & \sum_{\lambda \in I_1} \left[ g^0_{n+1, \mu +p^n\lambda},    
         \sum_{j=0}^r \left( p^j \sum_{i=j}^r c_i \binom{i}{j}(-\lambda)^{i-j} \right) X^{r-j} Y^j \right], \\
 \label{T^-}  T^-([g^0_{n,\mu},v]) & = & \left[ g^0_{n-1, [\mu]_{n-1}},    
         \sum_{j=0}^r \left( \sum_{i=j}^r p^{r-i} c_i {i \choose j} 
         \left( \frac{\mu - [\mu]_{n-1}}{p^{n-1}} \right)^{i-j} \right) X^{r-j} Y^j \right]  (n > 0),\quad \\
  \label{T^-0}  T^-([g^0_{n,\mu},v]) & = &  [ \alpha,  \sum_{j=0}^r  p^{r-j}  c_j  X^{r-j} Y^j ]   (n=0). 
\end{eqnarray}
These explicit formulas for $T^+$ and $T^-$ will be used to compute $(T-a_p)f$ for the functions $f \in \mathrm{ind}_{KZ}^G \Sym^r\bar\Q_p^2$.

%
%
Next let us define, for $0\leq i\leq b$, and $0\leq m<p-1$, the sums
\begin{eqnarray}
S_{r,i,m}&:=&\underset{\substack{j\equiv b-m\mod (p-1)\\0\leq j<r-m}}\sum {j\choose i}{r\choose j}\\
&=& \underset{\substack{j\equiv b-m\mod (p-1)\\i\leq j<r-m}}\sum {r\choose i}{r-i\choose j-i}\\
&=& {r\choose i}\cdot\left( \underset{\substack{j\equiv b-m\mod (p-1)\\i\leq j\leq r}}\sum {r-i\choose j-i}-{r-i\choose r-m-i}\right)\\
\label{A0}&=& \tilde{S}_{r,i,m}- {r\choose i}{r-i\choose m},
\end{eqnarray}
where $\tilde{S}_{r,i,m}:=\underset{\substack{j\equiv b-m\mod (p-1)\\0\leq j\leq r}}\sum {j\choose i}{r\choose j}$.

 With this notation, we state the following technical lemma.

\begin{lemma}\label{cong2}Let $r=b+sp^t(p-1)$ with $p\nmid s$, so that $t=v(r-b)$.
For $0\leq i<b$ and $0\leq m<p-1$, one has
$$S_{r,i,m}\equiv {r\choose i}\left({b-i\choose m}-{r-i\choose m}\right)\mod p^{t+1}\equiv 0\mod p^t.$$
\end{lemma}

\begin{proof}
Define
\begin{eqnarray}
\nonumber g_{r,i,m}(x)&:=& \frac{x^{i-(b-m)}}{i!} \cdot \frac{d}{dx^i}(1+x)^r\\ 
\label{A}&=& 
x^{i-(b-m)}{r\choose i} \cdot (1+x)^{r-i}\\
\nonumber &=& \quad   \sum_{0\leq j \leq r-i}{r\choose i}{r-i\choose j}x^{j+i-(b-m)}\\
\label{AA}&=& \sum_{i\leq l \leq r}\quad {r\choose l}  {l\choose i}x^{l-(b-m)},   \quad \text{ with }  l=i+j
\end{eqnarray}
Evaluating  \eqref{AA} at $x=\zeta$ and taking sum over all $\zeta\in\mu_{p-1}$,
\begin{eqnarray*}
\underset{\zeta\in\mu_{p-1}}\sum g_{r,i}(\zeta)&=& \sum_{i\leq l \leq r}\quad {r\choose l}  {l\choose i} \underset{\zeta\in\mu_{p-1}}\sum \zeta^{l-(b-m)}\\
&=& \sum_{\substack{i\leq l \leq r\\l\equiv b-m\mod (p-1)}}\quad {r\choose l}  {l\choose i} (p-1)
=(p-1)\tilde{S}_{r,i,m}
\end{eqnarray*}
Hence by \eqref{A}, we have 
\begin{eqnarray}
 {r\choose i}\underset{\zeta\in\mu_{p-1}}\sum \zeta^{i-(b-m)}(1+\zeta)^{r-i}&=&(p-1)\tilde{S}_{r,i,m}\\
\label{B}\implies {r\choose i}\underset{\zeta\in\mu_{p-1}}\sum \zeta^{i-(b-m)}(1+\zeta)^{b-i}(1+\zeta)^{r-b}&=&(p-1)\tilde{S}_{r,i,m}.
\end{eqnarray}
If $\zeta\neq -1$, then $(1+\zeta)^{r-b}=(1+\zeta)^{(p-1)sp^t}=(1+pz_\zeta)^{sp^t}\equiv 1\mod p^{t+1}$ and so
\begin{equation}\label{C}
(p-1)\tilde{S}_{r,i,m}\equiv {r\choose i}\cdot \underset{\zeta\in\mu_{p-1}\setminus\{-1\}}\sum \zeta^{i-b+m}(1+\zeta)^{b-i}
={r\choose i}\cdot B\mod p^{t+1},
\end{equation}
where $B:=\underset{\zeta\in\mu_{p-1}\setminus\{-1\}}\sum \zeta^{i-b+m}(1+\zeta)^{b-i}$ only depends on $b, m $ and $i$ 
(not on $s$ or $t$). 

Putting $r=b$ 
 in \eqref{B}, we get 
\begin{eqnarray*}
{b\choose i}\cdot B&=&(p-1)\tilde{S}_{b,i,m}=(p-1){b\choose b-m}{b-m\choose i}\\
&\implies & B=(p-1){b-i\choose m}.
\end{eqnarray*}
Hence from \eqref{C}, we obtain 
\begin{eqnarray*}
\nonumber (p-1)\tilde{S}_{r,i,m}&\equiv&(p-1){r\choose i}{b-i\choose m}\mod p^{t+1}\\
{\overset{\eqref{A0}}\implies} S_{r,i,m}=\tilde{S}_{r,i,m}- {r\choose i}{r-i\choose m} &\equiv&{r\choose i}\left({b-i\choose m}-{r-i\choose m}\right)\mod p^{t+1}.
\end{eqnarray*}
We also conclude that $$S_{r,i,m}\equiv 0\mod p^t,$$ as ${b-i\choose m}-{r-i\choose m}\in \frac{(r-b)}{m!}\Z\subset (r-b)\Z_p$, for $m\leq p-1$.
\end{proof}

\begin{prop}\label{mainProp}
Let $a_p\in m_{\bar\Z_p}$ be fixed. 
Assume that $r=k-2\equiv b\mod p^t(p-1)$, such that  $2v(a_p)\leq b\leq p-1$. In the case $2v(a_p)=b$, 
further assume $b$ is odd and that $ p^{-b/2} a_p \not\equiv \pm 1 \mod m_{\bar\Z_p}$.
If $t>2v(a_p)-1$, then there is a surjection $$\mathrm{ind}_{KZ}^G (V_r/V_r^{(1)})\twoheadrightarrow \bar\Theta_{k,a_p}.$$
\end{prop}
\begin{proof}
With the notation as in \eqref{chain}, we will show that $P(M_1)=0$, by showing $P(M_n)=0$, $P(M_{n-1}/M_n)=0,
\cdots$,  and finally  $P(M_1/M_2)=0$, so the map $P$ factors through
$$P: \frac{M_0}{M_1}=\mathrm{ind}_{KZ}^G \left(\frac{V_r}{V_r^{(1)}}\right)\twoheadrightarrow \bar\Theta_{k,a_p}.$$
For each $m$ with $1\leq m\leq n=\lfloor v(a_p)\rfloor$, we define the element $f=f_0+f_1\in 
\mathrm{ind}_{KZ}^G\Sym^r\bar\Q_p^2$, as follows:
\begin{eqnarray}
 f_0 &=& \left[1, \quad \frac{(p-1)p^m}{a_p^2}\cdot\left(\underset{\substack{0\leq j<r-m\\ j\equiv b-m\mod (p-1)}}\sum{r \choose j}X^{r-j}Y^j\right)\right],\\
 f_1 &=& \left[g_{1,0}^0,\:(1-p){r\choose m}\cdot \frac{F_m(X,Y)}{a_p}\right]+
\underset{\lambda\in \F_p^*}\sum\left[g_{1,[\lambda]}^0, \: \left(\frac{p}{[\lambda]}\right)^m\cdot\frac{F_{0}(X,Y)}{a_p}\right],\quad
\end{eqnarray}
where $F_j(X,Y)$ are the polynomials as defined in \eqref{F_m}. 
Let us compute how the operators $T^+$ and $T^-$ 
act on parts of the function $f$. 

We note that $(p^m/a_p^2).p^{m+p-1}$ is integral and vanishes in characteristic $p$, 
as we have  assumed 
that $2v(a_p)\leq b\leq p-1$, and $m\geq 1$. By formula \eqref{T^-0}, $T^-f_0$ vanishes.
%
\begin{eqnarray*}
T^+f_0&{\overset{\eqref{T^+}}=}& \underset{\lambda\in \F_p}\sum\left[g_{1,[\lambda]}^0,\underset{j=0}{\overset{r}\sum}
\frac{p^{j+m}(p-1)}{a_p^2}
\underset{\substack{j\leq i<r-m\\i\equiv b-m\mod (p-1)}}\sum{r\choose i}{i\choose j}(-[\lambda])^{i-j}X^{r-j}Y^j\right]\\
&\equiv& \underset{\lambda\in \F_p}\sum\left[g_{1,[\lambda]}^0,\underset{j=0}{\overset{b-1}\sum}
\frac{p^{j+m}(p-1)}{a_p^2}
\underset{\substack{j\leq i<r-m\\i\equiv b-m\mod (p-1)}}\sum{r\choose i}{i\choose j}(-[\lambda])^{i-j}X^{r-j}Y^j\right]\mod\wp\\
&&\hspace{3.8in}\text{ as } b+m> 2v(a_p),\\
&\equiv& \underset{\lambda\in \F_p^*}\sum\left[g_{1,[\lambda]}^0,\underset{j=0}{\overset{b-1}\sum}
\frac{p^{j+m}(p-1)}{a_p^2}(-[\lambda])^{b-m-j} S_{r,j,m}
X^{r-j}Y^j\right]\\
&& \quad +[g_{1,0}^0, \frac{p^b(p-1)}{a_p^2}{r\choose b-m} X^{r-(b-m)}Y^{b-m}]\\
&\equiv &[g_{1,0}^0, \frac{p^b(p-1)}{a_p^2}{r\choose b-m} X^{r-(b-m)}Y^{b-m}]\mod\wp,
\end{eqnarray*}
by Lemma \ref{cong2}, as we know that $t+j+m-2v(a_p)\geq t+m-2v(a_p)\geq t-(2v(a_p)-1)>0$.
Hence $T^+f_0$ is integral for $b\geq 2v(a_p)$ and 
vanishes in characteristic $p$ if $b>2v(a_p)$.

%
%

Using the formula \eqref{T^-} on $f_1$, we compute that
\begin{eqnarray*}
T^-f_1&\equiv& \left[\mathrm{Id}, \:  \frac{p^m}{a_p}(1-p){r\choose m}X^mY^{r-m}\right]\\
&& +\left[\mathrm{Id},\:   \underset{j=0}{\overset{r}\sum}\frac{p^m}{a_p}{r\choose j}\left(\underset{\lambda\in\F_p^*}\sum([\lambda])^{r-j-m}\right) X^{r-j}Y^j\right]\mod (p^{r-b+m-v(a_p)})\\
&\equiv & \left[\mathrm{Id}, \: \frac{p^m}{a_p}(1-p){r\choose m}X^mY^{r-m}
                           \:\: +\underset{\substack{0\leq j\leq r\\j\equiv b-m\mod (p-1)}}\sum \frac{p^m}{a_p}(p-1){r\choose j}X^{r-j}Y^j\right]\mod \wp\\
 &= &     \left[\mathrm{Id}, \: 
                      \underset{\substack{0\leq j<r-m \\ j\equiv b-m\mod (p-1)}}\sum \frac{p^m}{a_p}(p-1){r\choose j}X^{r-j}Y^j\right]   \\
 &=& a_pf_0,                                        
\end{eqnarray*}
so that  $T^-f_1-a_pf_0$ is integral and  vanishes in characteristic $p$.

 We note that for $0\leq i<b-m$,
$$v\left({r-m\choose i}-{b-m\choose i}\right)=v\left({b-m+p^t(p-1)s\choose i}-{b-m\choose i}\right)\geq t,$$
and using the formula \eqref{T^+}, we conclude that $T^+f_1$ is integral and vanishes mod $\wp$, since 
$t>2v(a_p)-1>v(a_p)$ (one may assume $v(a_p)>1$, as for $v(a_p)\leq 1$ the result already follows from 
\cite{[BG09], Bhattacharya-Ghate-Rozensztajn16}).

Finally we note that for $m\geq 1$, we have 
$$-a_pf_1\equiv \left[g_{1,0}^0,\: -{r\choose m}F_m(X,Y)\right]\mod p.$$
Now considering all the components of $(T-a_p)f$, we know that it is integral and reduces to
$\left[g_{1,0}^0,\: -{r\choose m}F_m(X,Y)\right]$ modulo $\wp$, if $b>2v(a_p)$. We note that 
${r\choose m}$ is a $p$-adic unit, as $r\equiv b\mod p$ by the hypothesis, and $m\leq\lfloor v(a_p)\rfloor <b$.

However, if $b=2v(a_p)$ is odd, then we get
$$(T-a_p)f\equiv \left[g_{1,0}^0,\: -{r\choose m}F_m(X,Y)
-\frac{p^b}{a_p^2}\cdot {r\choose b-m}X^{r-(b-m)}Y^{b-m}\right]\mod\wp,$$
which is also integral.
By Rem. 4.4 of \cite{[BG09]}, we know that there is some function $f'$ such that $(T-a_p)f'$ is
integral and reduces to $\frac{p^b}{a_p^2} {r\choose b-m}\cdot[g_{1,0}^0, X^mY^{r-m}]$, 
for $m\leq \lfloor v(a_p)\rfloor<v(a_p)$. 
Therefore $$(T-a_p)(f+f')\equiv \left[g_{1,0}^0,\: -\left({r\choose m}-\frac{p^b}{a_p^2}\cdot {r\choose b-m}\right)F_m(X,Y)\right]\mod\wp.$$
Now, under the assumption $v(r-b)=t>2v(a_p)-1>0$, and for $m\leq \lfloor v(a_p)\rfloor \leq b$, 
the constant above reduces to ${b\choose b-m}\frac{p^b}{a_p^2}-{b\choose m}={b\choose m}\left(\frac{p^b}{a_p^2}-1\right)\mod\wp$, which is a unit if and only if $\frac{a_p}{p^{b/2}}\not\equiv \pm 1\mod \wp$.

In any case, multiplying by a unit, we conclude that under the hypothesis of the proposition $[g_{1,0}^0, F_m(X,Y)]\xmapsto{P} 0\in \bar\Theta_{k,a_p}$.  
Also for all $m\leq \lfloor v(a_p)\rfloor$, we  have $2m<b$.
Therefore it follows from Lemma 
\ref{polynomial}(b) that $[g_{1,0}^0, F_m(X,Y)]$ generates $M_m/M_{m+1}=\mathrm{ind}_{KZ}^G\left( \frac{V_r^{(m)}}{V_r^{(m+1)}}\right)$ over $G$. 
As the map $P$ is $G$-linear, we have $P(M_m/M_{m+1})=0$ for all $m=1,\cdots, \lfloor v(a_p)\rfloor$, and thus $\bar\Theta_{k,a_p}$ must be a quotient of $\mathrm{ind}_{KZ}^G\left(V_r/V_r^{(1)}\right)$.
\end{proof}

\begin{thm}\label{thmending}
Under the hypotheses of Proposition \ref{mainProp}, we have $$\bar V_{k,a_p}\cong  \mathrm{ind}_{G_{\Q_{p^2}}}^{G_{\Q_p}}(\omega_2^{b+1})\cong \bar V_{b+2, a_p}.$$
\end{thm}

\begin{proof}
By Proposition \ref{mainProp}, we have a surjection 
$\mathrm{ind}_{KZ}^G (V_r/V_r^{(1)})\twoheadrightarrow 
\bar\Theta_{k,a_p}.$ By Prop. 2.1 of \cite{ Bhattacharya-Ghate}, and since the compact induction is an exact functor, we have a short exact sequence
$$0\rightarrow \mathrm{ind}_{KZ}^G V_b\rightarrow\mathrm{ind}_{KZ}^G (V_r/V_r^{(1)})\rightarrow \mathrm{ind}_{KZ}^G (V_{p-1-b}\otimes D^b)\rightarrow 0.$$
Let the image inside $\bar\Theta_{k,a_p}$ of the submodule $ \mathrm{ind}_{KZ}^G V_b$  be denoted by $F_1$. Then the quotient $F_2:=\bar\Theta_{k,a_p}/F_1$ must factor through  $\mathrm{ind}_{KZ}^G (V_{p-1-b}\otimes D^b)$, and we have the following commutative diagram.
\begin{equation*}
    \xymatrix{0 \ar[r] & \mathrm{ind}_{KZ}^G V_b\ar@{>>}[d]\ar[r] & \mathrm{ind}_{KZ}^G (V_r/V_r^{(1)})\ar@{>>}[d] \ar[r] &
       \mathrm{ind}_{KZ}^G (V_{p-1-b}\otimes D^b)\ar@{>>}[d]\ar[r] & 0 \\
               0 \ar[r] & F_{1}\ar[r]  & \bar{\Theta}_{k,a_p}\ar[r] & F_2 \ar[r] & 0 .}
 \end{equation*}
By  (4.5) in \cite{ Glover}, the submodule $V_b$ of $V_r/V_r^{(1)}$ is generated by the image of the monomial $Y^r$  over $\Gamma$.
Since $v(a_p)>0$, we also know that $\mathrm{ind}_{KZ}^G\langle Y^r\rangle$ maps to $0\in\bar\Theta_{k,a_p}$, using Remark 4.4 in \cite{[BG09]}. 
Thus we conclude $F_1=0$ and $\bar\Theta_{k,a_p}$ is a quotient of
 $\mathrm{ind}_{KZ}^G (V_{p-1-b}\otimes D^b)$. 
 Since  $\bar\Theta_{k,a_p}$ lies in the image of mod $p$ LLC, it must be 
 isomorphic to the supercuspidal representation $\pi(p-1-b,0,\omega^{b})$,
 which is in correspondence with the irreducible Galois representation
  $\mathrm{ind}_{G_{\Q_{p^2}}}^{G_{\Q_p}}(\omega_2^{b+1})$. 
 By Theorem 2.6 of \cite{Edixhoven92} this is isomorphic to $\bar V_{b+2,a_p}$ and that completes our proof.
\end{proof}
 
\section*{Acknowledgements}
During this work the author was supported by PBC post-doctoral fellowship and afterwords by the post-doctoral research grant from MPIM, Bonn. The author thanks Prof. E. Ghate for introducing to the problem of reduction mod $p$ and for the numerous useful discussions on the subject.

  \end{document}